\documentclass[a4paper]{article}
\usepackage{amsmath}
\usepackage{amsfonts}
\usepackage{amssymb}
\usepackage{amsthm}
\usepackage[pdfborder={0 0 0 [0 0 ]},bookmarksdepth=3, bookmarksopen=true,unicode=true,pdfencoding=unicode]{hyperref}
\usepackage{mathtools}
\usepackage{tikz}
\usepackage{tikz-cd}
\usepackage{enumerate}
\usepackage[capitalize]{cleveref}
\usetikzlibrary{patterns}

\newcommand{\PP}{\ensuremath{\mathbb{P}}}
\newcommand{\bdy}{\ensuremath{\partial}}
\newcommand{\iso}{\ensuremath{\cong}}
\newcommand{\Z}[1][]{\ensuremath{\mathbb{Z}_{#1}}}
\newcommand{\Q}{\ensuremath{\mathbb{Q}}}
\newcommand{\R}{\ensuremath{\mathbb{R}}}

\newcommand{\N}{\ensuremath{\mathbb{N}}}

\newcommand{\F}{\ensuremath{\mathbb{F}}}

\newcommand{\gpp}[2]{\ensuremath{\langle #1\,|\, #2\rangle}}

\hypersetup{
    colorlinks=true,
    linkcolor=black,
    citecolor=black,
    filecolor=black,
    urlcolor=black,
}

\newtheorem{theorem}{Theorem}[section]

\newtheorem{prop}[theorem]{Proposition}
\newtheorem{conj}[theorem]{Conjecture}
\newtheorem{lem}[theorem]{Lemma}
\newtheorem{clly}[theorem]{Corollary}

\theoremstyle{definition}
\newtheorem{defn}[theorem]{Definition}

\theoremstyle{remark}
\newtheorem{rmk}[theorem]{Remark}

\theoremstyle{plain}
\newcounter{introthmcount}
\newtheorem{introthm}{Theorem}

\theoremstyle{definition}

\newcounter{dawidcomments}

\newcounter{garethcomments}

\usetikzlibrary{decorations.markings}
\usetikzlibrary{shapes.geometric}
\newcommand{\rv}[2][l]{\ensuremath{\mathcal{#2}_{#1}}}

\newcommand{\ttx}{ \texttt{$\textup{x}$}}
\newcommand{\tty}{ \texttt{$\textup{y}$}}
\newcommand{\ttz}{ \texttt{$\textup{z}$}}
\newcommand{\Fox}[2]{\ensuremath{\frac{\bdy r_{#1}}{\bdy \ttx_{#2}}}}
\newcommand{\FoxB}[2]{\ensuremath{{\bdy r_{#1}}/{\bdy \ttx_{#2}}}}
\newcommand{\Nov}{\ensuremath{\widehat{\Q G}^\phi}}
\newcommand{\Gr}{\ensuremath{{\rm Gr}_{m-1,n}}}
\newcommand{\cC}{\ensuremath{{\cal C}}}

\def\s-{\smallsetminus}
\DeclareMathOperator{\supp}{\mathrm{supp}}

\title{$\ell^2$-Betti numbers and coherence of random groups}
\author{Dawid Kielak, Robert Kropholler, and Gareth Wilkes}
\begin{document}
\maketitle
\begin{abstract}
We study $\ell^2$-Betti numbers, coherence, and (virtual) fibring of random groups in the few-relator model. In particular, random groups with negative Euler characteristic are coherent, have $\ell^2$-homology concentrated in dimension 1, and embed in a virtually free-by-cyclic group with high probability. In the case of Euler characteristic zero, we use Novikov homology to show that a random group is free-by-cyclic with positive probability.
\end{abstract}
\section{Introduction}
A frequent theme in mathematics is the study of `generic' objects---submanifolds in general position, equations without multiple roots, etc. In the field of group theory the concept of genericity is expressed in the terminology of {\em random groups}---generally speaking, by exploring which properties of a finitely presented group with relators of length $l$ become overwhelmingly prevalent as $l$ grows.

Several generic properties of 1-relator groups are known: in \cite{SS10}, Sapir and \v{S}pakulov\'a proved that a random 1-relator group with at least 3 generators embeds in an ascending HNN extension of a free group, and hence is residually finite \cite{BS05} and coherent \cite{FH99}. Earlier, Dunfield and Thurston \cite{DT06} had shown that the same is true for 2-generator 1-relator groups with positive probability.

Since then, ground-breaking results in the study of special groups by Wise \cite{Wise04}, Agol \cite{Agol13} and others have greatly increased our understanding of small-cancellation groups, which include random groups with few relators. In particular, residual finiteness is now known to hold for all random few relator groups.

In this paper we aim to push coherence results beyond groups with one relator to random groups of positive deficiency---that is, with more generators than relators. We both incorporate notions and techniques from the Dunfield--Thurston and Sapir--\v{S}pakulov\'a papers and use the powerful consequences of specialness. A key additional element in the strategy is the use of Novikov and $\ell^2$-homology, particularly in relation to a virtual fibring condition established by the first author in \cite{Kielak19}.

The distinction between groups with two generators and more than two generators seen above persists in this paper as a distinction between groups of deficiency 1 and those with deficiency greater than 1.

\begin{introthm}\label{thm_deficiency1}
For each $n\geqslant 2$, a random group with $n$ generators and deficiency 1 is a free-by-cyclic group
with positive asymptotic probability.
\end{introthm}


One may also consider a version of Theorem \ref{thm_deficiency1} in which the conclusion `free-by-cyclic' is replaced by `virtually free-by-cyclic'. This version is conjectured to hold with high probability, but our methods only yield a positive asymptotic probability. Theorem \ref{thm_deficiency1} as stated will be true with probability strictly less than 1 (see \cite{DT06} for this phenomenon with 2 generators).

We note that in the two-generator one-relator case, the `virtual' theorem is readily seen to hold with asymptotic probability 1, as noted in Corollary \ref{clly_2Gen1Rel} below. 
\begin{introthm}\label{thm_deficiencymore}
For each $n\geqslant 3$, a random group $G$ with $n$ generators and deficiency $d>1$ has the following properties:
\begin{itemize}
\item $b^{(2)}_1(G)=d-1$, and $b^{(2)}_i(G)=0$ for all $i\neq 1$;
\item $G$ embeds in a virtually free-by-cyclic group; and
\item $G$ is coherent;
\end{itemize}
with asymptotic probability 1.
\end{introthm}

Another way to phrase the dependence on deficiency is to consider the Euler characteristic. Theorem \ref{thm_deficiency1} and Theorem \ref{thm_deficiencymore} are the cases of Euler characteristic $\chi =0$ and $\chi < 0 $ respectively. The connection between Euler characteristic and coherence has been postulated elsewhere; for example, in a survey of coherence written by Wise \cite{Wise20}. Our work provides a positive answer to \cite[Conjecture 17.15(3)]{Wise20} with high probability and to \cite[Conjectures 17.14(2) and 17.15(2)]{Wise20} with positive asymptotic probability.

The structure of the paper is as follows. In \cref{sec_Prelims} we will review the formal definitions and background material required for the paper. In Section \ref{sec_MinimumConditions} we will introduce the key condition on a few-relator group on which all the other results rely, and show the consequence of these results regarding $\ell^2$-Betti numbers and fibring. Section \ref{sec_EmbeddingSCGroups} gives the construction used to embed the higher deficiency groups into groups of deficiency one. Finally in Section \ref{sec_Probability} we will build on the results of \cite{SS10} to show that our conditions hold with the required probability.
\linebreak

{\noindent \bf Acknowledgements.} The first author was supported by a grant from the German Science Foundation (DFG) within the Priority Programme SPP2026 `Geometry at Infinity'. The third author was supported by a Junior Research Fellowship from Clare College, Cambridge, and wishes to thank Alexander Roberts for reading the proof of Lemma \ref{lem_ManyCorners}. The authors wish to thank Giles Gardam for pointing out the remark following Proposition \ref{Prop:2genincoh}, and the referee for observing that the previous version of \cref{thm_deficiency1}  could be improved to its current formulation.

\section{Preliminaries}\label{sec_Prelims}
Throughout, we fix a generating set $\ttx_1,\ldots, \ttx_n$ and let $F_n$ be the free group thereon.

\subsection{Few-relator random groups}
The results in this paper concern the classical few-relator model introduced by Arzhantseva and Ol'shanskii \cite{AO96}. For a fixed $n\geqslant 2$ and $m\geqslant 1$, we consider for $l\geqslant 1$ the set $\rv{R}$ of (ordered) $m$-tuples $r_1,\ldots, r_m$ of cyclically reduced words of length $l$ in the alphabet $\ttx_1,\ldots, \ttx_n$. We consider the uniform probability measure on \rv{R}.
This yields a notion of a `random presentation' \[\gpp{\ttx_1,\ldots, \ttx_n}{r_1,\ldots, r_m}\]
with $n$ generators and $m$ relators. By a standard abuse of nomenclature this is also referred to as a `random group'.

A property ${\cal P}$ of groups (with $n$ generators and $m$ relators) holds {\em with asymptotic probability 1} (or {\em with high probability, w.h.p.}) if
\[\PP\left(\gpp{\ttx_1,\ldots, \ttx_n}{r_1,\ldots, r_m}\text{ has }{\cal P} \right) \to 1 \quad \text{as }l\to\infty.\]
The property $\cal P$ holds {\em with positive asymptotic probability} if
\[ \liminf_{l\to \infty} \PP\left(\gpp{\ttx_1,\ldots, \ttx_n}{r_1,\ldots, r_m}\text{ has }{\cal P} \right) > 0.\]
It is easy to see that if ${\cal P}_1,\ldots, {\cal P}_k$ are finitely many group properties, and each has asymptotic probability 1, then the property that all ${\cal P}_i$ hold simultaneously also has asymptotic probability 1. Similarly, if ${\cal P}_1,\ldots, {\cal P}_k$ have asymptotic probability 1 and ${\cal P}_0$ has positive asymptotic probability, then ${\cal P}_0,\ldots, {\cal P}_k$ hold simultaneously with positive asymptotic probability.

The {\em deficiency} of a finite group presentation $ \gpp{\ttx_1,\ldots, \ttx_n}{r_1,\ldots, r_m}$ is the quantity $n-m$. The {\em deficiency} ${\rm def}(G)$ of a group $G$ is the maximum deficiency of a finite presentation of $G$. Since there is a clear bound ${\rm def}(G)\leqslant b_1(G)$, and since a random group with $n$ generators and $m$ relators has
\[b_1(G) = \max(n-m, 0)\]
with high probability\footnote{This is a `well-known' folklore result which seems to have escaped explicit statement in the papers that the authors are aware of. One may deduce it from results like \cite[Theorem 3.2]{SS10} which state that random relators limit onto a Brownian motion, so that $m$ of them span an $m$-dimensional subspace of $H_1(F_n,\R)$ with high probability. A more elementary argument can be found in \cite{Wilkes19}.}, a random group with $n$ generators and $m\leqslant n$ relators has deficiency exactly $n-m$ with high probability. It thus makes sense to speak of `random groups with deficiency $d$'.

If $G$ is a group of deficiency $n-m$, and the presentation complex given by a presentation $ \gpp{\ttx_1,\ldots, \ttx_n}{r_1,\ldots, r_m}$ is aspherical, then we have $\chi(G) = 1-(n-m) = 1- {\rm def}(G)$. Since presentation complexes of random groups are aspherical with high probability (see Gromov \cite{Gromov93} as quoted below), we may also speak of, for example, `random groups of negative Euler characteristic'.

Many powerful properties of random groups flow from small-cancellation conditions.
\begin{defn}
A tuple of cyclically reduced words $r_1,\ldots,r_m$ satisfies the {\em small cancellation condition} $C''(\lambda)$, for $\lambda>0$, if whenever a reduced word $w$ or its inverse occurs as a cyclic subword both of $r_i$ and $r_j$ (or more than once in a single $r_i$), then $|w|<\lambda |r|$, where $|\cdot|$ denotes word length.
\end{defn}
We abuse terminology by saying that group presentations (and groups) satisfy $C''(\lambda)$ if their tuples of defining relators satisfy $C''(\lambda)$.

This is a slight variant of the usual $C'(\lambda)$ small cancellation: $C''(\lambda)$ is equivalent to `the presentation is $C'(\lambda)$ and no relator is a proper power'. This in particular forbids torsion from our groups. Random groups have no relators which are proper powers with high probability, so our variant condition serves mainly to make statements more concise.

We collect below various theorems from which we may deduce properties of random few-relator groups.
\begin{theorem}[Gromov \cite{Gromov93}, Section 9.B]\label{thm_SmallCancWHP}
For every $\lambda>0$, a random few-relator group satisfies $C''(\lambda)$ with high probability.
\end{theorem}
For more in-depth studies of small cancellation in the few-relator model see \cite[Lemma 3]{AO96}, including estimates of the asymptotics of the probability.
\begin{theorem}[See \cite{Gromov87}, Section 4.7 and \cite{Ol12}, Theorem 13.3]\label{ThmGromovHyp}
A group satisfying $C''(1/6)$ is hyperbolic, has aspherical presentation complex and is torsion-free.
\end{theorem}

\begin{theorem}[Wise \cite{Wise04}]
A finitely presented $C''(1/6)$ group acts properly discontinuously and co-compactly on a CAT(0) cube complex.
\end{theorem}
\begin{theorem}[Agol \cite{Agol13}]
A hyperbolic group acting properly and co-compactly on a CAT(0) cube complex is virtually special.
\end{theorem}
\begin{theorem}[Haglund--Wise \cite{HW08}, Agol \cite{Agol08}]
A virtually special group is virtually RFRS.
\end{theorem}
\begin{theorem}[Schreve \cite{Schreve14}]
A virtually special group satisfies the Atiyah conjecture.
\end{theorem}
We combine these into one statement for easier reference later.
\begin{theorem}\label{thm_C1/6consequences}
A finitely presented $C''(1/6)$ group is hyperbolic, 2-dimensional, virtually RFRS and satisfies the Atiyah conjecture. Thus a few-relator random group has these properties with high probability.
\end{theorem}

The notions of cube complexes, special groups, RFRS, etc.\ play no role in this paper other than as stepping stones between theorems, and we shall not trouble to define them. The Atiyah conjecture is discussed in the next section.

\subsection{\texorpdfstring{$\ell^2$}{l\texttwosuperior}-homology, skew-fields and Novikov rings}

A key tool in the proof of \cref{thm_deficiencymore} is the theory of $\ell^2$-homology. We give here a brief description of the theory, following L\"uck~\cite{Luck13}, but focusing only on the algebraic viewpoint.

Let $G$ be any discrete  group; let $X$ be its classifying space, and let $C_\bullet$ be the cellular chain complex of the universal covering of $X$. Letting $\mathcal N(G)$ denote the von Neumann algebra of $G$, we define $b_n^{(2)}(G)$, the \emph{$n^{th}$ $\ell^2$-Betti number of $G$},  to be the \emph{extended} von Neumann dimension of $H_n (\mathcal N(G) \otimes C_\bullet)$.

In fact, given any CW-complex $Y$ with an action of $G$ we may define its $\ell^2$-Betti number to be the extended von Neumann dimension of  $H_n (\mathcal N(G) \otimes C_\bullet(Y))$ where $C_\bullet (Y)$ denotes the cellular chain complex of $Y$. If $G$ is torsion free, the \emph{Atiyah conjecture} is the statement that if $Y$ has finitely many $G$-orbits of cells in every dimension and the action of $G$ on $Y$ is free, then the $\ell^2$-Betti numbers of $Y$ are all integers. In L\"uck's book this is referred to as the strong Atiyah conjecture over $\Q$; there is also a version of this conjecture for groups with \emph{bounded torsion}, that is groups whose torsion has uniformly bounded order.

Instead of tensoring $\Z G$ chain complexes with $\mathcal N(G)$, we may as well tensor with the \emph{Ore localisation} $\mathrm{Ore}(\mathcal N(G))$ of $\mathcal N(G)$ (with respect to the set of all non-zero-divisors of $\mathcal N(G)$). This new ring $\mathrm{Ore}(\mathcal N(G))$ is also known as the algebra of operators affiliated to $\mathcal N(G)$. One can extend the notion of von Neumann dimension to this setting rather easily, since Ore localisations are flat extensions. Now, since we are only interested in $\Z G$ chain complexes, one does not need the entire Ore localisation, but only the \emph{rational closure} of $\Z G$ inside $\mathrm{Ore}(\mathcal N(G))$; this last object is the \emph{Linnell ring}, and is denoted by $\mathcal D(G)$. The key point for us is that Linnell~\cite{Linnell1993} proved that the Atiyah conjecture (for torsion-free groups) is equivalent to $\mathcal D(G)$ being a skew-field, and the $\ell^2$-Betti numbers of $Y$ are in fact $\mathcal D(G)$ dimensions of the homology groups $H_n (\mathcal D(G) \otimes C_\bullet(Y))$ (and hence in particular are all integers). This applies to the situation when $Y$ is the universal cover of a classifying space of $G$, and thus yields
\[
 b_n^{(2)}(G) = \dim_{\mathcal D(G)} H_n (G; \mathcal D(G))
\]
provided that $G$ satisfies the Atiyah conjecture and is torsion free.

\medskip

Let $G$ be a group with presentation
\[\gpp{\ttx_1,\ldots, \ttx_n}{r_1,\ldots, r_m}\]
and assume that this presentation is $C''(1/6)$, so that $G$ is torsion free and satisfies the Atiyah conjecture.

Let $X$ be the associated presentation complex, and let $\widetilde X$ be the universal cover of $X$. The cellular chain complex of $\widetilde X$ over \Z\ takes the form
\[\begin{tikzcd}
0\ar{r} & (\Z G)^m \ar{r}{A} & (\Z G)^n \ar{r}{B} & \Z G
\end{tikzcd}\]
where $B$ is the right multiplication of a row vector in $(\Z G)^n$ by the matrix
\[B = \begin{pmatrix}
\ttx_1-1 & \cdots & \ttx_n-1
\end{pmatrix}^{T} \]
and $A$ is the right multiplication by the Jacobian matrix
\[A = \begin{pmatrix}
\Fox{1}{1}  & \cdots & \Fox{1}{n}\\
\vdots & \ddots & \vdots\\
\Fox{m}{1} & \cdots & \Fox{m}{n}
\end{pmatrix} \]
whose entries are the Fox derivatives of the $r_i$---see \cite{Fox53}.

As the presentation complex for our $C''(1/6)$ presentation is aspherical, the above complex of $G$-modules provides a free resolution of $\Z$ and thus may be used to compute the (co)homology of $G$. In particular, the chain complex
\[\begin{tikzcd}
0\ar{r} & ({\cal D}(G))^m \ar{r}{{\cal D}(G)\otimes A} & ({\cal D}(G))^n \ar{r}{{\cal D}(G)\otimes B} & {\cal D}(G)
\end{tikzcd}\]
computes the $\ell^2$-homology $H_i(G;{\cal D}(G))$.

Our primary concern in Section \ref{sec_MinimumConditions} will be to find conditions ensuring that the map ${\cal D}(G)\otimes A$ is {\em injective}, so that the second $\ell^2$-Betti number
\[b^{(2)}_2(G) = \dim_{{\cal D}(G)}(\ker({\cal D}(G)\otimes A))\]
vanishes. If the group $G$ is 2-dimensional, then once the second $\ell^2$-Betti number is known, we may also find---for example, by way of the Euler characteristic \cite[Theorem 1.35]{Luck13}---the value of $b^{(2)}_1(G)$. In particular, if $b^{(2)}_2(G) = 0$ then $b^{(2)}_1(G)=n-m-1$.

\begin{theorem}\label{thm_ConsequenceOfL2}
Let $G$ be a group with a finite $C''(1/6)$ presentation of deficiency 1.
\begin{enumerate}
 \item If $b^{(2)}_2(G)=0$ then $G$ is virtually free-by-cyclic and therefore coherent.
 \item If there exists an epimorphism $\phi \colon G \to \Z$ with finitely generated kernel, then $G$ is free-by-cyclic.
\end{enumerate}
\end{theorem}
\begin{proof}
	By Theorem \ref{thm_C1/6consequences}, the group $G$ is virtually RFRS and of dimension 2. By the discussion above, we have
	\[
	b^{(2)}_2(G)=b^{(2)}_1(G)=0
	\]
By \cite[Theorem 5.4]{Kielak19}, there is a finite-index subgroup $G'$ of $G$ which admits a map $G'\twoheadrightarrow \Z$ with kernel of type FP${}_2$. We may now apply a theorem of Gersten \cite{Gersten96} to see that the kernel is hyperbolic and hence finitely presented. Finally, we may deduce that the kernel is free by \cite{Bieri76}. A virtually free-by-cyclic group is coherent by \cite{FH99}.

\smallskip
Now suppose that we have an epimorphism $\phi \colon G \to \Z$ with $K = \ker \phi$ finitely generated. We then have $b^{(2)}_1(G)=0$ by a result of L\"uck~\cite[Theorem 1.39]{Lueck2002}, and therefore $b^{(2)}_2(G)=0$ as well, as above. We conclude that $G$ is coherent from (1), and so the group $K$ is finitely presented. Now we argue precisely as above, and conclude that $K$ is free.
\end{proof}

We end this section by noting that for random groups with 2 generators, the (in)coherence of a group is known with high probability.
\begin{clly}\label{clly_2Gen1Rel}
With high probability, a random 2-generator 1-relator group $G$ is virtually free-by-cyclic.
\end{clly}
\begin{proof}	
By Theorem \ref{thm_C1/6consequences}, the presentation is $C''(1/6)$ with high probability. The $\ell^2$-Betti numbers of a torsion-free 2-generator 1-relator group all vanish by \cite{DL07}, and we may now apply the theorem.
\end{proof}
\begin{rmk}
It is worth comparing this with the earlier result of Dunfield and Thurston \cite{DT06}: a random 2-generator 1-relator group is {\em fibred} with asymptotic probability strictly between 0 and 1 \cite[Theorem 6.1]{DT06}, but is {\em virtually} fibred with high probability.
\end{rmk}
\begin{prop}\label{Prop:2genincoh}
Let $G$ be a random group with two generators and $m\geqslant 2$ relators. Then with high probability $b^{(2)}_1(G)=0$ and $G$ is incoherent.
\end{prop}
\begin{proof}
With high probability the presentation complex is aspherical and $G$ is infinite, virtually RFRS, and satisfies the Atiyah conjecture. Consider the chain complex
\[\begin{tikzcd}
0\ar{r} & ({\cal D}(G))^m \ar{r}{{\cal D}(G)\otimes A} & ({\cal D}(G))^2 \ar{r}{{\cal D}(G)\otimes B} & {\cal D}(G)
\end{tikzcd}\]
which computes the $\ell^2$-Betti numbers. Since $G$ is infinite we have $b^{(2)}_0(G)=0$, so ${\cal D}(G)\otimes B$ is surjective and hence $\ker ({\cal D}\otimes B)$ has ${\cal D}(G)$-dimension 1. But since the presentation complex was aspherical, $A$ is injective and hence ${\cal D}(G)\otimes A$ is not the zero map, and its image has dimension at least 1---and must therefore equal $\ker ({\cal D}(G)\otimes B)$. It follows that $b^{(2)}_1(G)=0$.

By \cite[Theorem 5.3]{Kielak19} the group $G$ is virtually fibred. However the finitely generated fibre in question cannot be finitely presented---if it were, then by \cite{Bieri76} the fibre is free and we find $\chi(G)=0$, a contradiction. Hence $G$ is incoherent.
\end{proof}
\begin{rmk}
It is worth remarking that this result also holds in the density model of random groups at densities $d<1/6$ (keeping the number of generators at $2$). The key conditions are that the random group $G$ is infinite, the given presentation complex is aspherical, and that $G$ acts properly and cocompactly on a CAT(0) cube complex (whence is virtually RFRS and satisfies the Atiyah conjecture, by theorems previously cited). Theorem \ref{ThmGromovHyp} applies at all densities $d<1/2$, and random groups are cubulated at densities $d<1/6$ by a theorem of Ollivier and Wise \cite{OW11}.
\end{rmk}
\begin{conj}
A random group with $n$ generators and $m$ relators has
\[b^{(2)}_1(G) = \max(n-m-1, 0) \]
with high probability, hence is coherent with high probability if $\chi(G)\leqslant 0$ and incoherent with high probability if $\chi(G)>0$.
\end{conj}
\section{Minimum conditions and $\ell^2$-homology}\label{sec_MinimumConditions}
\subsection{The minimum condition}
\begin{defn}
Let $r$ be a word of length $l$ on $\ttx_1,\ldots, \ttx_n$. Let $C_r$ denote the cycle of length $l$, viewed as a graph with $l$ edges of length $1$; we endow $C_r$ with a basepoint and an orientation. We further label the edges by  ${\ttx_1}^{\pm 1},\ldots, {\ttx_n}^{\pm 1}$ in such a way that reading the labels starting from the basepoint and following the orientation yields $r$. In this way vertices of $C_r$ become identified with prefixes of the word $r$.

For a homomorphism $\phi\colon F_n \to \Z$ with $\phi(r)=0$, consider the induced piecewise affine continuous map $\phi_r\colon C_r\to \R$ that sends vertices to the value under $\phi$ of the corresponding prefixes. The {\em lower section} of $r$ is the subset
\[L_\phi(r)=\phi_r^{-1}(\min(\phi_r(C_r)))\subseteq C_r.\]
We will also speak of the vertices and edges of the lower section as being {\em $\phi$-minimal}.
\end{defn}
\begin{defn}\label{def_MinimumCondition}
Let $(r_1,\ldots,r_m)$ be a tuple of cyclically reduced words over $\ttx_1,\ldots, \ttx_n$ (where $n>m$) and let $\phi\colon F_n\to \Z$. The pair $((r_1,\ldots, r_m),\phi)$ {\em satisfies the minimum condition} if, possibly after reordering the $\ttx_i$, 
the following properties hold:
\begin{itemize}
\item $\phi(r_i) = 0$ for all $i$;
\item for all $i$ the lower section $L_\phi(r_i)$ consists either of exactly one vertex, whose two adjacent edges are labelled by $\ttx_i^\pm$ and $\ttx_n^\pm$, or of exactly one edge, labelled by $\ttx_i^\pm$, whose adjacent edges are both labelled by $\ttx_n^\pm$.
\end{itemize}
The pair $((r_1,\ldots, r_m),\phi)$ satisfies the {\em standard} minimum condition if, in addition, we have $\phi(\ttx_i) \geqslant 0$ for all $i<n$ and $\phi(\ttx_n) <0$. Note that $\phi(\ttx_n)=0$ is automatically forbidden by the definition.

Finally, we say that $(r_1, \ldots, r_{m})$ {\em satisfies a (standard) minimum condition} if some $\phi$ exists such that $((r_1,\ldots, r_m),\phi)$ satisfies the (standard) minimum condition.
\end{defn}

One may dually define the {\em (standard) maximum condition} by reversing all the signs in Definition \ref{def_MinimumCondition}.
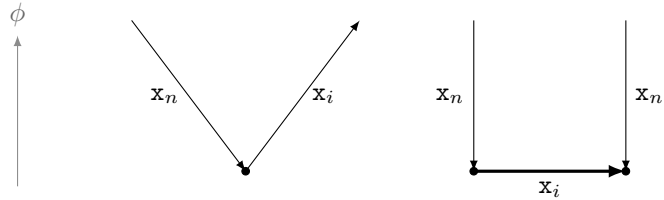
\begin{figure}[ht]
\centering
\begin{tikzpicture}[>=latex]
\draw[->, gray!90] (-2,-0.2) -- (-2, 1.8);
\draw[gray!120] (-2,1.8) node[above]{$\phi$};
\draw[<-] (1,0) -- (-0.5,2) node[midway, left]{$\ttx_n$};
\draw[->] (1,0) -- (2.5,2) node[midway, right]{$\ttx_i$};
\filldraw (1,0) circle (0.05);

\begin{scope}[shift={(4,0)}]
\draw[->] (0,2) -- (0,0) node[midway, left]{$\ttx_n$};
\draw[->] (2,2) -- (2,0) node[midway, right]{$\ttx_n$};
\draw[->, very thick] (0,0) -- (2,0) node[midway, below]{$\ttx_i$};
\filldraw (0,0) circle (0.05);
\filldraw (2,0) circle (0.05);
\end{scope}
\end{tikzpicture}
\caption{Allowed minima of $r_i$ with respect to $\phi$ in the standard minimum condition. There is no imposition on the orientation of $r_i$ in this diagram---that is, the minimum  may occur at/near either an $\ttx_i$ edge or an $\ttx^{-1}_i$ edge. Non-standard minimum conditions allow for the reversal of arrows (with respect to $\phi$).}
\end{figure}

\begin{rmk}\label{rmk_ReplaceWithStandard}
Suppose that $((r_1,\ldots, r_m),\phi)$ satisfies the minimum condition. Then by possibly performing substitutions of the form $\ttx_i \mapsto \ttx_i^{-1}$, one obtains a pair  $((r'_1,\ldots, r'_m),\phi')$ satisfying the {\em standard} minimum condition. Since one naturally has an isomorphism of groups
\[\gpp{\ttx_1,\ldots, \ttx_n}{r_1,\ldots, r_m} \iso \gpp{\ttx_1,\ldots, \ttx_n}{r'_1,\ldots, r'_m} \]
one may always assume that the minimum condition is the standard minimum condition when deducing group-theoretic properties. Moreover, the homomorphism induced by $\phi$ on the first group agrees with that defined by $\phi'$ on the second one.
\end{rmk}

The key usage of the minimum condition in this article will be the following theorem.

\begin{theorem}\label{thm_MinCondImpliesL2}
Suppose that $((r_1,\ldots, r_m),\phi)$ satisfies the minimum condition. Suppose that \[G=\gpp{\ttx_1,\ldots, \ttx_{n}}{r_1,\ldots, r_m}\]
is torsion free, satisfies the Atiyah conjecture, and that the images of the generators $\ttx_i$ in $G$ are all non-trivial.
Then $b_2^{(2)}(G)=0$

If additionally $n=m+1$, and $((r_1,\ldots, r_m),\phi)$ satisfies the maximum condition, then the epimorphism $G \to \Z$ induced by $\phi$ is \emph{algebraically fibred}, that is, its kernel is finitely generated.
\end{theorem}

\subsection{Novikov rings}

In the presence of a homomorphism $\phi\colon G\to\Z$, with kernel $K$, and a choice of section $s\colon \Z\to G$, $s(1)=t$, we may identify the group ring $\Q G$ with a {\em twisted group ring} $(\Q K)\Z=(\Q K)[t,t^{-1}]$ as in \cite[Section 2.2]{Kielak19} (twisted group rings are usually called \emph{crossed products} in the literature). For our purposes we may consider this as a `semi-direct product' of a group ring by a group. Each $g\in G$ may be written uniquely as
\[g = k_g t^{\phi(g)}\]
for some $k_g \in K$. Elements in the group ring $\Q G$ thus acquire the structure of a polynomial in $t$ and $t^{-1}$ with coefficients in $\Q K$. Addition is the usual addition of polynomials, whereas the ring multiplication is twisted via the formula
\begin{equation}(k_g t^{\phi(g)}) (k_h t^{\phi(h)}) = \left(k_g {k_h}^{t^{-\phi(g)}}\right) t^{\phi(g)+\phi(h)} = k_{gh}t^{\phi(gh)}.\label{eq_TwistedConvolution}\end{equation}

The minimum condition of the previous section is chosen with the following property in mind.
\begin{lem}\label{lem_FoxDerivsOfMinimumCondition}
Suppose that $((r_1,\ldots, r_m),\phi)$ satisfies the standard minimum condition and choose some $t\in G$ such that $\phi(t)=1$. Define
\[P_i = \min \phi_r(C_r).\]
Then
\[\Fox{i}{i} = k_i t^{P_i} + \sum_{p > P_i} \nu_{i,i,p}t^p, \quad \Fox{i}{j} = \sum_{p > P_i} \nu_{i,j,p}t^p \]
for some $k_i\in K = \ker(G\stackrel{\phi}{\longrightarrow} \Z)$ and some $\nu_{i,j,p} \in \Q K$.
\end{lem}
\begin{proof}
Recall---for example, from Formula 2.5 in \cite{Fox53}---that the Fox derivative \FoxB{i}{j} acquires a term $+u$ for each occurence of $\ttx_j$ in $r_i$ such that $r_i = u\ttx_j v$, and a term $-u\ttx_j^{-1}$ for each occurence of $\ttx_j^{-1}$ in $r_i$ where $r_i = u\ttx_j^{-1} v$. No initial subword of $r_i$ evaluates under $\phi$ to a value smaller than $P_i$, so that no \FoxB{i}{j} has any $t^p$-term for any $p<P_i$.

A $t^{P_i}$-term can only arise from an initial segment $u$ with $\phi(u)= P_i$---either where $u$ ends with a letter $\ttx_j^{-1}$ or is followed by a letter $\ttx_j^{+1}$. By the standard minimum condition, there are either one or two initial segments of $r_i$ attaining the minimal $\phi$-value, depending whether the lower section is a vertex or an edge. Note that the standard minimum condition gives one of four cases, depending on whether the lower section is a vertex or an edge, and on the relative orientations of $r_i$ and the relevant edge $\ttx_i$ at the minimum:
\begin{itemize}
\item $\phi(\ttx_i)>0$, and there is a unique initial segment $u=u'\ttx_n$ of $r_i$ such that $\phi(u)=P_i$ and $r_i=u'\ttx_n\ttx_i v$.
\item $\phi(\ttx_i)>0$, and there is a unique initial segment $u=u'\ttx_i^{-1}$ of $r_i$ such that $\phi(u)=P_i$ and $r_i=u'\ttx_i^{-1}\ttx_n^{-1}v$.
\item $\phi(\ttx_i)=0$, and the are exactly two initial segments $u=u'\ttx_n$ and $u\ttx_i$ of $r_i$ such that $\phi(u)=\phi(u\ttx_i)=P_i$, and $r_i=u'\ttx_n\ttx_i\ttx_n^{-1} v$.
\item $\phi(\ttx_i)=0$, and the are exactly two initial segments $u=u'\ttx_n\ttx_i^{-1}$ and $u'\ttx_n$ of $r_i$ such that $\phi(u)=\phi(u'\ttx_n)=P_i$, and $r_i=u'\ttx_n\ttx_i^{-1}\ttx_n^{-1} v$.
\end{itemize}
In all four cases, there is precisely one $t^{P_i}$-term $u = k_ut^{P_i}=k_it^{P_i}$ occuring in \FoxB{i}{i}, and no $t^{P_i}$-term in any other \FoxB{i}{j}.
\end{proof}
\begin{rmk}
It is worth reconsidering Definition \ref{def_MinimumCondition} in light of this previous result. If our definition of `minimum condition' had been simply that each relator has a unique $\phi$-minimal vertex, then it could well be the case that $r_1$ and $r_2$ have minima of the form $r_1=u_1\ttx_1^{-1} \ttx_2 v_1$ and $r_2=u_2\ttx_1^{-1} \ttx_2 v_2$---so that both $\FoxB{1}{1}$ and $\FoxB{1}{2}$ have $t^{P_1}$-terms, and both $\FoxB{2}{1}$ and $\FoxB{2}{2}$ have $t^{P_2}$-terms. As will be seen in the proof of Proposition \ref{prop_RowReduction}, such a circumstance is undesirable.

The special role of $\ttx_n$ in Definition \ref{def_MinimumCondition} guarantees that we may invert relators to obtain the standard minimum condition and thus obtain Lemma \ref{lem_FoxDerivsOfMinimumCondition}.
\end{rmk}
Elements of the group ring having the special form taken by \FoxB{i}{i} in Lemma \ref{lem_FoxDerivsOfMinimumCondition} are especially amenable to computation in a certain ring containing $\Q G$, the {\em Novikov ring}. We give here a simplified statement which is less broad than the most general definition, but is sufficient for our needs. Proposition 3.8 of \cite{Kielak19} gives the equivalence of our statement with the `true' definition.

Novikov rings are treated in detail in the notes of Sikorav~\cite{Sikorav}.

\begin{defn}
Let $G$ be a group and let $\phi\colon G\to \Z$ be a homomorphism with kernel $K$. The {\em Novikov ring} \Nov\ of $G$ with respect to $\phi$ is the ring of formal series
\[\sum_{p=P}^\infty \nu_p t^p\]
for $P\in \Z$ and $\nu_p \in \Q K$. Addition is termwise, and multiplication is given by the linear extension of the twisted convolution formula \eqref{eq_TwistedConvolution}.
\end{defn}
The Novikov ring is thus the ring of formal power series in $(\Q K)[[t,t^{-1}]]$ which are `infinite only in the positive direction'---that is, there is some $P$ such that $\nu_p = 0$ for all $p< P$. 
Observe that the Novikov ring contains $\Q G$, and hence multiplication turns it into a $\Q G$-module.

Note that every element of $x  = \sum_p \mu_p t^p \in \Nov$ has a \emph{support}
\[
\supp x = \{ p \in \Z \mid \mu_p \neq 0\}
\]
which is a subset of $\Z$ bounded from below.

In the following, we will identify a matrix $A$ with the linear map $x \mapsto xA$ given by right-multiplication by $A$.

The matrices discussed below are very closely related to what Sikorav calls \emph{trivially invertible matrices}; the relation is given by multiplication with a diagonal matrix whose diagonal terms are of the form $kt^P$ with $k \in K$.

\begin{prop}\label{prop_RowReduction}
Let $A = (a_{ij})$ be an $m \times n$ matrix over \Nov, where $n\geqslant m$. Suppose that for every $i \leqslant m$ we have $k_i \in K$ and $P_i \in \Z$ such that $a_{ii} - k_i t^{P_i}$ and $a_{ij}$ with $j \neq i$ are supported over $\Z \cap [P_i + 1, \infty)$. 
Then the following hold.
\begin{enumerate}[(1)]
\item $A \colon ({\Nov}{})^m \to ({\Nov}{})^{n}$ is injective.
\label{item 1}
\item When $n = m+1$, the map $A \colon ({\Nov}{})^m \to ({\Nov}{})^{m+1}$ followed by a projection of $({\Nov}{})^{m+1}$ onto $({\Nov}{})^{m}$ forgetting the last coordinate is surjective.
\label{item 4}
\item When $A$ lies over $\Q G$, then $A \colon (\Q G)^m \to (\Q G)^n$ is injective.
\label{item 2}
\item When $G$ satisfies the Atiyah conjecture and is torsion free, and $A$ lies over $\Q G$, then $\mathcal D(G) \otimes A \colon \mathcal D(G)^m \to {\cal D}(G)^n$ is injective.
\label{item 3}
\end{enumerate}
\end{prop}
\begin{proof}
Note that it is sufficient to prove each of the statements for the $m\times m$ top submatrix of $A$ (in \eqref{item 4} we then show that $A$ is surjective). We will therefore assume $A$ to be a square matrix.

%
Multiplying each row of $A$ by a unit in $\Q G$ does not affect its injectivity or surjectivity in any of the situations in the proposition. Therefore, for each $i$ we may multiply the $i^{th}$ row by $(k_i t^{P_i})^{-1}$, and hence assume that $k_i = 1$ and $P_i = 0$.

Observe that the matrix $B = A - I$ has every entry supported over a subset of the positive integers. Thus, for every $k \in \N$ the matrix $(-B)^k$ has entries supported over subsets of $\Z \cap [k, \infty)$, and therefore the matrix
\[
C = \sum_{k = 0}^\infty (-B)^k
\]
is a well-defined matrix over \Nov. Moreover, we have $AC = (I + B)C = I$: for every $K$, we have
\[AC = (I+B)\sum_{k=0}^{K-1} (-B)^k + (I+B)\sum_{k=K}^{\infty}(-B)^k = I - (-B)^K + (I+B)\sum_{k=K}^{\infty}(-B)^k\]
so $AC-I$ is supported over $\Z\cap [K,\infty)$ for every $K$, and thus vanishes. It is also true that $CA = I$, for the same reason. We conclude that
\[A \colon (\Nov{})^m \to (\Nov{})^m\]
is an isomorphism, and hence it is in particular injective and surjective, proving \eqref{item 1} and \eqref{item 4}.

\smallskip
When $A$ is defined over $\Q G$, then we may treat it as a linear map $A \colon (\Q G)^m \to (\Q G)^m$. Injectivity follows immediately from \eqref{item 1}. 

\smallskip
To prove \eqref{item 3}, we need to note that there exists a skew-field $\F(G)$ fitting into the following commutative diagram.
\[\begin{tikzcd}
\Q G \ar[hook]{r} \ar[hook]{d}& {\cal D}(G) \ar[hook]{d}\\
\Nov \ar[hook]{r} & \F (G)
\end{tikzcd}\]
The skew-field was constructed in \cite[Lemma 5.5]{FK18}; it is the Malcev--Neumann completion of $\mathcal D(K)[t^{-1},t]$ with respect to the obvious biordering on $\Z = \langle t \rangle$, which is exactly the same as the ring of twisted Laurent series in one variable $t$ with coefficients $\mathcal D(K)$.

The inverse $C$ we constructed lies over \Nov, and hence $\F(G) \otimes C$ (tensored over the Novikov ring) is an inverse of $\F(G)\otimes A$. Hence, $\F(G)\otimes A$ is injective. This implies that the restriction of $\F(G)\otimes A$ to $\mathcal D(G)^m$ is also injective. But this restriction is precisely ${\mathcal D(G)\otimes A}$.
 \end{proof}

The last ingredient in the proof of \cref{thm_MinCondImpliesL2} is a connection between the Novikov rings and the notion of algebraic fibring, as defined in the statement of the theorem.

 \begin{theorem}[Sikorav~\cite{Sikorav1987}]
 \label{Sikorav}
Let $G$ be a finitely generated group. An epimorphism $\phi \colon G \to \Z$ is algebraically fibred if and only if the homology groups $H_1(G; \widehat{\Q G}^\phi)$ and $H_1(G; \widehat{\Q G}^{-\phi})$ both vanish.
 \end{theorem}

\begin{proof}[Proof of Theorem \ref{thm_MinCondImpliesL2}]
By Remark \ref{rmk_ReplaceWithStandard} we may assume, without changing the isomorphism type of $G$, that $((r_1,\ldots, r_m),\phi)$ satisfies the standard minimum condition for some $\phi\colon F_n\to \Z$. By Lemma \ref{lem_FoxDerivsOfMinimumCondition}, the Jacobian matrix
\[{A=(\FoxB{i}{j})}\]
has entries satisfying the conditions of Proposition \ref{prop_RowReduction}.

Note that the first group homology of $G$ is computed by the chain complex
\[
 C_2 \to C_1 \to C_0
\]
of free $\Z G$-modules, where $C_2 = \Z G^m, C_1 = \Z G^{n}, C_0 = \Z G$, and where we understand both differentials: the differential $C_2 \to C_1$ is equal to the matrix $A$; the entries of the differential $C_1 \to C_0$ are of the form $\ttx_i-1$, where $\ttx_i$ ranges over all generators. Since the generators are non-trivial in $G$ by assumption, the entries of this last differential are non-zero.

By \cref{prop_RowReduction}\eqref{item 3}, the differential $A$ is injective, and hence the chain complex above computes all homology groups of $G$.
By \cref{prop_RowReduction}\eqref{item 3}, the differential $A$ is still injective over the Linnell skew-field $\mathcal D(G)$, and hence $b_2^{(2)}(G) = 0$.

\smallskip
Now suppose that $n = m+1$.
To compute the first Novikov homology $H_1(G;\Nov)$, we tensor the above complex over $\Z G$ with $\Nov$. Let $c$ be a $1$-cycle in this tensored complex. \cref{prop_RowReduction}\eqref{item 4} tells us that we may add a $2$-boundary to $c$ and have all of its entries but the last one equal to $0$. But the last entry of the differential $C_1 \to C_0$ is not zero in $\Z G$, and therefore also non-zero in $\Nov$. This implies that the last entry of $c$ is a zero-divisor in $\Nov$, since $c$ is a $1$-cycle. 

We are assuming that $G$ satisfies the Atiyah conjecture, and therefore $\Q G$ has no non-trivial zero-divisors. It is an easy exercise to see that the same then holds for $\Nov$. We conclude that $c$ is a $2$-boundary, and hence $H_1(G;\Nov)=0$.

If $((r_1,\ldots, r_m),\phi)$ satisfies the maximum condition as well, we repeat the argument for $-\phi$, and conclude using \cref{Sikorav} that $\phi$ is algebraically fibred.
\end{proof}

\section{Embeddings of small cancellation groups}\label{sec_EmbeddingSCGroups}

In this section we build injective homomorphisms from high deficiency groups into deficiency 1 groups, which will allow us to pass results about coherence to the high deficiency groups. 
\begin{theorem}\label{thm_SmallCancEmbedding}
Let $((r_1,\ldots, r_m),\phi)$ be a tuple satisfying the standard minimum condition, for $(m\leq n-2)$. Then there exist cyclically reduced words $s_1,\ldots, s_m$ on an alphabet $\tty_1,\ldots, \tty_{m+1}$, and a map $\psi\colon \{\tty_1,\ldots, \tty_{m+1}\}\to\Z$ such that $((s_1,\ldots, s_m),\psi)$ satisfies the standard minimum condition, and an injection
\[\gpp{\ttx_1,\ldots, \ttx_n}{r_1,\ldots, r_m} \hookrightarrow \gpp{\tty_1,\ldots, \tty_{m+1}}{s_1,\ldots, s_m}.\]
If $((r_1,\ldots, r_m),\phi)$ also satisfies the standard maximum condition, the $s_i$ can be chosen to also satisfy the standard maximum condition.

Suppose in addition the $r_i$ satisfy the $C''(1/(6+\epsilon))$ small cancellation condition for some $\epsilon>0$ and that all $r_i$ have length $l$ where $l> 12 + 72/\epsilon$. Then the $s_i$ can be chosen to satisfy $C''(1/6)$. 
\end{theorem}
The tool we use to produce embeddings is the congruence extension property of small cancellation theory. It is implied by stronger theory in \cite{Ol95} concerning general hyperbolic groups; see also \cite[Theorem 3.5]{Sapir09} for a self-contained argument.
\begin{lem}[Congruence extension property]\label{lem_CongExtProperty}
Let $r_1,\ldots, r_m$ be words in an alphabet $\ttx_1,\ldots, \ttx_n$, and let $w_1,\ldots, w_n$ be words in an alphabet $\tty_1,\ldots, \tty_k$ satisfying $C'(1/12)$. Then the map $f\colon \ttx_i\mapsto w_i$ induces an {\em injective} homomorphism 
\[\gpp{\ttx_1,\ldots, \ttx_n}{r_1,\ldots,r_m} \to \gpp{\tty_1,\ldots, \tty_k}{f(r_1),\ldots, f(r_m)}.\]
\end{lem}
With the tool in hand, let us return to \cref{thm_SmallCancEmbedding}.
\begin{proof}[Proof of Theorem \ref{thm_SmallCancEmbedding}]
Consider a free group $F_{m+1}$ on an alphabet $\tty_1,\ldots, \tty_{m}, \tty_{m+1}$. For added clarity in the proof we rename $\tty_{m+1} =\ttz$. We now define elements of $F_{m+1}$, dependent on a large integer $N\gg\max(|\phi(\ttx_i)|)$. First we define
\begin{eqnarray*}
w_1 &=&  \begin{aligned}[t] \ttz^N \tty_1 \ttz^{N-1} \tty_1 \ttz^{N-2} \cdots \ttz \tty_1 \ttz^{-1} \cdots \ttz^{1-N}\tty_1 \ttz^{-N} \tty_1 \ttz^{-\phi(\ttx_1)} \tty_1 \cdot \\  \ttz^{-N-1} \tty_1 \cdots \ttz^{-2N} \tty_1 \ttz^{2N} \cdots \ttz^{N+1} \tty_1\end{aligned} \\
w_2 &=&  \begin{aligned}[t] \ttz^N \tty_2 \ttz^{N-1} \tty_2 \ttz^{N-2} \cdots \ttz \tty_2 \ttz^{-1} \cdots \ttz^{1-N}\tty_2 \ttz^{-N} \tty_2 \ttz^{-\phi(\ttx_2)} \tty_2 \cdot \\  \ttz^{-N-1} \tty_2 \cdots \ttz^{-2N} \tty_2 \ttz^{2N} \cdots \ttz^{N+1} \tty_2 \end{aligned}  \\
&\vdots & \\
w_m &=&  \begin{aligned}[t] \ttz^N \tty_m \ttz^{N-1} \tty_m \ttz^{N-2} \cdots \ttz \tty_m \ttz^{-1} \cdots \ttz^{1-N}\tty_m \ttz^{-N} \tty_m \ttz^{-\phi(\ttx_m)} \tty_m \cdot \\  \ttz^{-N-1} \tty_m \cdots \ttz^{-2N} \tty_m \ttz^{2N} \cdots \ttz^{N+1} \tty_m\end{aligned}
\end{eqnarray*}
Next, we define
\begin{eqnarray*}
w_{m+1} &=& \begin{aligned}[t]  \ttz^N \tty_1^2 \ttz^{N-1} \tty_1^2 \ttz^{N-2} \cdots \ttz \tty_1^2 \ttz^{-1} \cdots \ttz^{1-N}\tty_1^2 \ttz^{-N} \tty_1^2 \ttz^{-\phi(\ttx_{m+1})} \tty_1^2 \cdot\hspace{2em} \\ \ttz^{-N-1} \tty_1^2 \cdots \ttz^{-2N} \tty_1^2 \ttz^{2N} \cdots \ttz^{N+1} \tty_1^2 \end{aligned}\\
&\vdots&\\
w_{n-1} &=&  \begin{aligned}[t] \ttz^N \tty_1^{n-m} \ttz^{N-1} \tty_1^{n-m} \ttz^{N-2} \cdots \ttz \tty_1^{n-m} \ttz^{-1} \cdots \ttz^{1-N}\tty_1^{n-m} \ttz^{-N} \tty_1^{n-m}\cdot\\ \ttz^{-\phi(\ttx_{n-1})} \tty_1^{n-m}  \ttz^{-N-1} \tty_1^{n-m} \cdots \ttz^{-2N} \tty_1^{n-m} \ttz^{2N} \cdots \ttz^{N+1} \tty_1^{n-m} \end{aligned}
\end{eqnarray*}
And finally, we define
\begin{eqnarray*}
w_n &= & \begin{aligned}[t] \ttz^{-N} \tty_1^{n-m+1} \ttz^{-N-1} \cdots \ttz^{2-2N}\tty_1^{n-m+1}\ttz^{2N-2}  \cdots  \ttz^N \tty_1^{n-m+1}\ttz^{-\phi(\ttx_n)} \cdot \\ \tty_1^{n-m+1} \ttz^{N-1}  \tty_1^{n-m+1} \ttz^{N-2}  \cdots  \ttz\tty_1^{n-m+1} \ttz^{-1}  \cdots \ttz^{1-N}\tty_1^{n-m+1}\end{aligned}
\end{eqnarray*}
Note that if $u$ is a cyclic subword of some $w_i^{\pm1}$ of length at least $2N+2(n-m)+1$, then $u$ must contain a subword $\ttz^{\pm 1} \tty_j^a \ttz^{\pm' 1}$. The number $j$ together with the sign and magnitude of $a$ diagnose a unique $w_i$ which contains $u$, and the direction in which $w_i$ is read. Note further that by inspection the longest repeated subword within any of $w_i$ is a subword $\ttz^{\pm(2N-2)}\tty_1^{n-m+1} \ttz^{\pm (2N-1)}$ repeated within $w_n$.

Hence repeated subwords among the words $w_i$ have linear length in $N$, while the length of each of the words $w_i$ is a quadratic polynomial in $N$. For any $0<\delta <1/12$ we may therefore choose $N$ large enough so that the words $w_i$ are $C''(\delta)$ and all have length in the range $[N^2(4-\delta), N^2(4+\delta)]$. 
\begin{figure}[ht]
\centering
\begin{tikzpicture}[scale=0.9,>=latex,decoration={
    markings,
    mark=at position 0.6 with {\arrow{>}}}]
\begin{scope}[gray!90]
\begin{scope}[shift={(-7.5,0)}]
\draw[->] (0,-3.5)--(0,4.5) ;
\end{scope}
\draw[dotted] (-6.3,0)--(4,0); 
\draw[dotted] (-4.5,-2)--(4,-2); 
\draw[dotted] (-4.5,-3)--(4,-3); 
\draw[dotted] (-4.5,1.5)--(-2.25,1.5); 
\draw[dotted] (4.5,1)--(2.25,1); 
\end{scope}
\begin{scope}[gray!120]
\draw (-7.5,4.5) node[above]{$\phi,\psi$};
\draw (-6.85,0) node[]{$\phi_{\rm min}$};
\draw (-6,-2) node[]{$\phi_{\rm min}-\frac{N(N-1)}{2}$};
\draw (-6,-3) node[]{$\phi_{\rm min}-\frac{N(N+1)}{2}$};
\end{scope}
\draw[postaction={decorate}] (-4.5,1.5) -- (0,0) node[near end,above]{$\ttx_n$};
\draw[postaction={decorate}] (0,0) -- (4.5,1) node[near start,above]{$\ttx_i$};

\begin{scope}[red]

\draw[postaction={decorate}] (-1.5, -2) -- (-0.75,-2) node[midway, below]{$\tty_1^{n-m+1}$} ;
\draw[postaction={decorate}] (-0.75,-2) .. controls (-0.6,-0.8) and (-0.15,-0.1).. (0,0);
\draw[postaction={decorate}] (-2.25,0) .. controls (-2.1,-0.1) and (-1.65,-0.8) .. (-1.5,-2);
\draw[postaction={decorate}] (-2.25,1.5) -- (-2.25,0) node[near end, left]{$\ttz^{-\phi(\ttx_n)}$};
\draw[postaction={decorate}] (-3.75, 3.5) -- (-3,3.5) node[near start, above]{$\tty_1^{n-m+1}$} ;
\draw[postaction={decorate}] (-3,3.5) .. controls (-2.85,2.3) and (-2.4,1.6).. (-2.25,1.5);
\draw[postaction={decorate}] (-4.5,1.5) .. controls (-4.35,1.6) and (-3.9,2.3) .. (-3.75,3.5);

\draw[dashed, postaction={decorate}] (-4.3, 4.5) -- (-0.2, 4.5) node[midway, above]{$w_n$};
\end{scope}


\begin{scope}[red]

\draw[postaction={decorate}] (0.75, -3) -- (1.5,-3) node[midway, below]{$\tty_i$} ;
\draw[postaction={decorate}] (0,0) .. controls (0.15,-0.1) and (0.65,-0.8) .. (0.75,-3);
\draw[postaction={decorate}] (1.5,-3) .. controls (1.6,-0.8) and (2.1,-0.1) .. (2.25,0);
\draw[postaction={decorate}] (2.25,0) -- (2.25,1) node[near start, right]{$\ttz^{-\phi(\ttx_i)}$};
\draw[postaction={decorate}] (3, 4) -- (3.75,4) node[midway, below]{$\tty_i$} ;
\draw[postaction={decorate}] (2.25,1) .. controls (2.4,1.1) and (2.9,1.8) .. (3,4);
\draw[postaction={decorate}] (3.75,4) .. controls (3.85,1.8) and (4.35,1.1) .. (4.5,1);

\draw[dashed, postaction={decorate}] (0.2, 4.5) -- (4.3, 4.5) node[midway, above]{$w_i$};
\end{scope}
\end{tikzpicture}
\caption{A schematic showing the replacement of $\ttx_i$ with $w_i$ at the minimum point of a relator $r_i$. The `parabolic' segments represent those parts of the words $w_i$ and $w_n$ not otherwise labelled. The vertical scale shows the variation of $\phi$ and $\psi$ over the words. The horizontal scale has no precise meaning.}
\label{fig_SmallCancellation}
\end{figure}
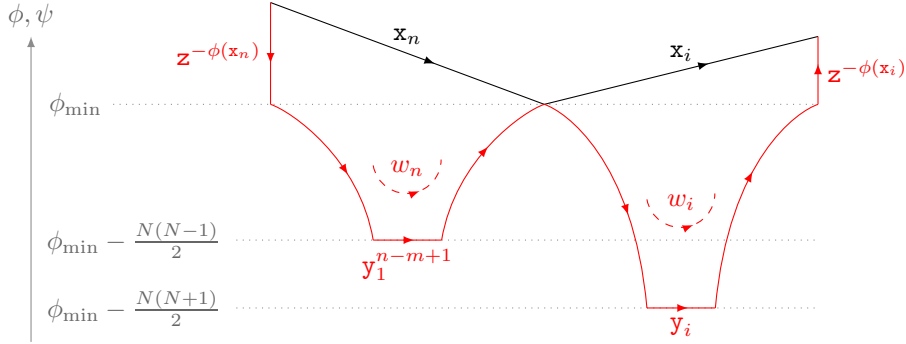

By Lemma \ref{lem_CongExtProperty}, we now have the desired embedding  
\[\begin{tikzcd}\gpp{\ttx_1,\ldots, \ttx_n}{r_1,\ldots,r_m} \ar[hook]{r}& \gpp{\tty_1,\ldots, \tty_k}{f(r_1),\ldots, f(r_m)}.\end{tikzcd}\]

Consider the map $\psi\colon F_{m+1}\to \Z$ given by $\ttz\mapsto -1, \tty_i\mapsto 0$ $(i\leq m)$. By construction we have $\psi(w_i) = \phi(\ttx_i)$ for all $i$. 

The minimum value attained by $\psi$ over initial segments of the word $w_i$, for $i<n$, is $-N(N+1)/2$, attained both by the initial segment ending in $\ttz^N$ and the initial segments ending $\ttz^N\tty_j^k$ (for whichever $j$ and $k$ appear). The minimum value attained along initial segments of $w_n$ is $-N(N-1)/2+\phi(\ttx_n)$, which is $N(N-1)/2$ lower than $\phi(\ttx_n)$.

The respective maximum values of $\psi$ along subwords of $w_i$ and $w_n$ are $N(3N+1)/2 + \phi(\ttx_i)$ and $(N-1)(3N-2)/2$.

Now the relator $r_i$ attains its minimal $\phi$-value $\phi_{\rm min}(i)$ either at a single vertex, whose adjoining edges are labelled $\ttx_i$ and $\ttx_n$ or a single edge labelled by $\ttx_i$ whose adjecent edges are labelled by $\ttx_n$. It follows that when we substitute the words $w_i$ for $\ttx_i$ in the relators $r_i$, the minimal $\psi$-value of the word $f(r_i)$ so obtained is given by \[\psi_{\rm min}(i)=\phi_{\rm min}(i)-N(N+1)/2,\] and this is attained exactly along a single edge labelled by $\tty_i$, which appears in a segment $\ttz^N \tty_i\ttz^{-N}$. See Figure \ref{fig_SmallCancellation}. Note that this minimum property is retained when $f(r_i)$ is cancelled down to a cyclically reduced word $s_i$. Hence $((s_1,\ldots, s_m),\psi)$ satisfies the standard minimum condition. If $((r_1,\ldots, r_m),\phi)$ additionally satisfies the standard maximum condition then $((s_1,\ldots, s_m),\psi)$ similarly satisfies the standard maximum condition.

Suppose the words $r_i$ satisfy the $C''(1/(6+\epsilon))$ small cancellation condition, and let $s_i$ be the word obtained from $f(r_i)$ by substituting the small cancellation words $w_i$ and cyclically reducing as above. Then $|s_i| \geqslant N^2l(4-3\delta)$. Any (cyclic) subword $u$ of length at least $|s_i|/6= N^2l(4-3\delta)/6$ must therefore contain the images in $s_i$ of at least 
\[\frac{N^2l(4-3\delta)}{6\cdot N^2(4+2\delta)} -2\]
letters of $r_i$. As $\delta\to 0$ we have
\[\frac{N^2l(4-3\delta)}{6N^2(4+2\delta)} -2\to \frac{l}{6}-2 > \frac{l}{6+\epsilon}\]
for the values of $l$ as in the theorem statement. For $\delta$ small enough, we find that the $C''(1/(6+\epsilon))$ condition on the $r_i$ implies the $C''(1/6)$ condition on the $s_i$ as desired.
\end{proof}

\section{Probability}\label{sec_Probability}
In this section we will apply methods similar to \cite{DT06} and \cite{SS10} to give probabilistic statements concerning the minimum conditions of Section \ref{sec_MinimumConditions}.

\begin{prop}\label{prop_Deficiency1Probability}
Let $G$ be a random group with $n$ generators and $n-1$ relators $r_1,\ldots r_{n-1}$. Then there exists $\phi\colon G\to \Z$ such that $((r_1,\ldots, r_{n-1}),\phi)$ satisfies both the minimum condition and the maximum condition with positive asymptotic probability.
\end{prop}
\begin{proof}
For a random tuple of cyclically reduced words $r_1,\ldots, r_{n-1}$, with high probability there is a unique map $\phi\colon F_n\twoheadrightarrow \Z$ up to sign such that $\phi(r_i)=0$ for all $i$. Let $\rv{R}'$ be the set of ordered $(n-1)$-tuples of relators of length $l$ such that the group defined by these relators has first Betti number 1. Then
\[|\rv{R}'| \sim |\rv{R}| \sim (2n-1)^{(n-1)l}\]
where $\sim$ denotes asymptotic equivalence
\[f(l)\sim g(l) \quad\Leftrightarrow \quad\frac{f(l)}{g(l)}\to 1 \text{ as }l\to \infty.\]
Define a function
\[\tau\colon \rv{R}' \to \rv[l+8]{R}'\]
in the following manner. Let $(r_1,\ldots, r_{n-1})\in \rv{R}'$ and let $\phi\colon F_n\twoheadrightarrow \Z$ be a map such that $\phi(r_i)=0$ for all $i$. Let $j$ be the first index such that $\phi(\ttx_j)\neq 0$. By changing the sign of $\phi$ we may assume that $\phi(\ttx_j)<0$. Note that these properties determine $\phi$ uniquely. For an index $1\leqslant i\leqslant n-1$ define
\[i'=\begin{cases} i & \text{if } i<j\\ i+1 & \text{if }i\geqslant j.\end{cases}\]
For each relator $r_i$, find the first $\phi$-minimal vertex along $C_{r_i}$, and form a new relator $r_i'$ by inserting a commutator $\ttx_j\ttx_{i'}^\epsilon\ttx_j^{-1}\ttx_{i'}^{-\epsilon}$. Here $\epsilon\in\{\pm1\}$ is chosen so that $\phi(\ttx_{i'}^\epsilon)\leqslant 0$; if $\phi(\ttx_{i'})= 0$ then choose $\epsilon = +1$ unless this causes $r'_i$ to not be a cyclically reduced word, in which case choose $\epsilon =-1$. Note that in no other circumstance can $r_i'$ not be cyclically reduced, for a failure of cyclic reduction would imply that we did not insert the commutator at a $\phi$-minimal point.

Now find the first $\phi$-maximal vertex along $C_{r'_i}$, and form a new relator $r_i''$ by inserting a commutator $\ttx_j^{-1}\ttx_{i'}^{-\epsilon}\ttx_j\ttx_{i'}^{\epsilon}$, where $i'$ and $\epsilon$ are chosen in the same manner as in the previous paragraph.

Note that insertion of commutators does not affect the abelianization of a group, and so the tuple $(r_1'',\ldots, r_{n-1}'')\in\rv[l+8]{R}'$ determines the same map $\phi$ as before.

The lower section $L_\phi(r_i'')$ is now, by construction, either a single vertex following $\ttx_{i'}^\epsilon$ or a single edge labelled by $\ttx_{i'}^\epsilon$. After relabelling so that $\ttx_j$ becomes $\ttx_n$, the pair $((r_1'',\ldots, r_{n-1}''),\phi)$ satisfies the minimum condition. Similarly it satisfies a maximum condition.

Note also that the map
\[\tau\colon (r_1,\ldots, r_{n-1}) \mapsto (r_1'',\ldots, r_{n-1}'')\]
is  injective: if $(r_1'',\ldots, r_{n-1}'')$ lies in the image of $\tau$, then taking the unique map $\phi$ as specified above we may find the unique $\phi$-maximal vertex or edge of $r_i''$ and remove a suitable commutator to recover $r_i'$, then also remove a commutator at the unique $\phi$-minimal vertex or edge of $r_i'$ to recover $r_i$.

Finally let \rv{S} denote the set of tuples in $\rv{R}'$ satisfying both a minimum condition and a maximum condition, so that $\rv[l+8]{S}\supseteq \tau(\rv{R}')$. We have
\[\frac{|\rv[l+8]{S}|}{|\rv[l+8]{R}|} \geqslant \frac{|\rv{R}'|}{|\rv[l+8]{R}|} \sim \frac{|\rv{R}|}{|\rv[l+8]{R}|} \sim \frac{1}{(2n-1)^{8(n-1)}} > 0\]
which concludes the proof.
\end{proof}
\begin{proof}[Proof of Theorem \ref{thm_deficiency1}]
By \cref{thm_C1/6consequences,prop_Deficiency1Probability}, a deficiency 1 group presentation is $C''(1/6)$ and satisfies both the minimum and the maximum conditions for some $\phi$ with positive probability.  By Theorem \ref{thm_MinCondImpliesL2}\footnote{The condition in Theorem \ref{thm_MinCondImpliesL2} that no $\ttx_i$ vanishes in $G$ is a consequence of $C''(1/6)$ once the relators have length at least 3, by Greendlinger's Lemma.} we know that $\phi$ is algebraically fibred. 
Finally, Theorem \ref{thm_ConsequenceOfL2} shows that $G$ is free-by-cyclic and coherent.
\end{proof}
This method of adding commutators to produce the minimum condition will also stand us in good stead in the case of presentations of deficiency at least two. Loosely speaking, at this deficiency there will be many essentially different possible maps to \Z, and for each of these we can add commutators to produce tuples satisfying a minimum condition.

With high probability, none of the generators $\ttx_i$ is killed by the map to the abelianization $H_1(G,\Q)$ of a random group $G$ of deficiency at least 2, whence the set of maps $\phi\colon F_n\to \Z$ with $\phi(r_i)=0$ for all $i$ and with some $\phi(\ttx_i)=0$ comprises only a finite union of proper subspaces of  $H^1(G,\Q)$. We shall assume this holds henceforth. This is a point of only minor significance, but will allow us, by peturbing maps slightly, to only consider maps not killing any $\ttx_i$---so that lower sections of relators will comprise a union of $\phi$-minimal vertices, and we will not need to worry about edges.

Let the set of {\em valid slopes} be
\[S = S(r_1,\ldots, r_m) = \{\phi\colon F_n\to \Z \mid \phi(r_i)=0\,\forall i,\; \phi(\ttx_i)\neq 0\,\forall i\}.\]
Define an equivalence relation $\sim$ on by
\[\phi\sim \phi' \quad\Leftrightarrow \quad L_\phi(r_i)=L_{\phi'}(r_i) \,\forall i.\]
Observe that the equivalence classes of $\sim$ are open subsets of $H^1(F_n,\Q)$, with the property that if $\phi\sim\phi'$ then $((r_1,\ldots, r_m),\phi)$ satisfies the minimum condition if and only if $((r_1,\ldots, r_m),\phi')$ does.

Our earlier heuristic statement concerning `many essentially different possible maps to \Z' has the following more precise formulation.
\begin{lem}\label{lem_ManyCorners}
Suppose the number of relators $m$ is strictly less than $n-1$. For every $K\in \N$ and each $\epsilon >0$,
\[\PP(|S/{\sim}| < K) <\epsilon\]
for all $l$ sufficiently large.
\end{lem}
\begin{proof}
Let $\cC_n$ be the space of continuous functions $f\colon [0,1]\to \R^n$ such that $f(0)=0$, equipped with the supremum norm. Let $A_{n,K}$ be the set of functions $f\in \cC_n$ such that the convex hull $\Delta(f)$ of the projection of $f$ to the hyperplane orthogonal to $f(1)$ is an $(n-1)$-dimensional polytope with at most $K-1$ vertices. Note that $f(1)\neq 0$ with Wiener measure 1, so we may ignore the measure zero set where $f(1) = 0$.

Let $T_0,T_1,\ldots, T_l$ be the non-backtracking random walk in $\Z^n$ corresponding to the cyclically reduced relator $r_1$ of length $l$. Let $X_l(t)\colon [0,1]\to \R^n$ be the piecewise affine function defined by $X_l(i/l) = T_i / \sqrt{l}$. By \cite[Theorem 3.3]{SS10}, $X_l(t)$ converges in distribution to the standard Brownian motion $Y(t)$ on $\R^n$.

The other $m-1$ relators $r_2,\ldots, r_m$ span with high probability an $(m-1)$-dimensional subspace $W_l(r_2,\ldots, r_m)$ of $\R^n$. We consider this to be a random variable $W_l$ with values in the Grassmannian \Gr\ of $(m-1)$-planes in $\R^n$. Again by \cite[Theorem 3.3]{SS10} the relators $r_i$ converge to independent Brownian motions, which in particular are spherically symmetric; it follows that $W_l$ converges in distribution to the uniform distribution $U$ on \Gr. Since $X_l$ and $W_l$ are independent, the pair $(X_l,W_l)$ converges in distribution to $(Y,U)$ on $\cC_n\times \Gr$.

For $(f,V)\in \cC_n\times\Gr$, we may consider the orthogonal projection $\pi_V(f) \in \cC_{n-m+1}$ of $f$ onto the plane\footnote{Strictly speaking, $\pi_V(f)$ is a function with values in $V^\perp$, not in $\R^{n-m+1}$, unless we choose an orthonormal basis for $V^\perp$. We will ignore this unimportant technicality for the sake of readability; it could be resolved by working on charts in \Gr\ on which a continuous choise of basis of $V^\perp$ may be made.} $V^{\perp}$ orthogonal to $V$, and ask whether $\pi_V(f)\in A_{n-m+1,K}$. We desire to show that $\pi_{W_l}(X_l)\notin A_{n-m+1,K}$ with high probability: for then there are at least $K$ distinct maps $\tilde\phi\colon\R^{n-m+1}\to \R$ whose sets of $\tilde\phi$-minimal vertices of $\Delta(\pi_{W_l}(X_l))$ are distinct. Composing these with the projections $\R^n\to \R^{n-m+1}$ gives $K$ maps $\R^n\to\R$ which, after perturbing them to be defined over \Q, witness the desired inequality $|S/{\sim}|\geqslant K$.

Since $n-m+1\geqslant 2$, by \cite[Lemma 3.4]{SS10} we know that $A_{n-m+1,K}$ is closed and has Wiener measure zero. Since the mapping $(f,V)\mapsto \pi_V(f)$ is continuous, the set $B=\{(f,V)\mid \pi_V(f)\in A_{n-m+1,K}\}$ is also closed. For each fixed $V\in \Gr$, we have
\begin{eqnarray*}
\mu_{\rm Wiener}\left(\pi_V^{-1}(A_{n-m+1,K})\right) &=& \PP\left(Y\in \pi_V^{-1}(A_{n-m+1,K})\right)\\
&=& \PP\big(\pi_V(Y)\in A_{n-m+1,K}\big)\\
&=& \mu_{\rm Wiener}\left(A_{n-m+1,K}\right) = 0\end{eqnarray*}
since the orthogonal projection $\pi_V(Y)$ of the standard Brownian motion $Y$ is again a standard Brownian motion, and hence is also governed by the Wiener measure.

Letting $\nu$ denote the product measure $\mu_{\rm Wiener}\times \mu_{\rm Gr}$, we deduce
\[\nu(B)= \iint \chi_{B} d\mu_{\rm Wiener} d\mu_{\rm Gr} = \int \mu_{\rm Wiener}\left(\pi_V^{-1}(A_{n-m+1,K})\right) dV = 0.\]
Finally, by the convergence in distribution and Portmanteau's Lemma, we find
\begin{eqnarray*}\limsup_{l\to\infty} \PP\left(\pi_{W_l}(X_l)\in A_{n-m+1,K}\right) &=& \limsup_{l\to\infty} \PP\left((X_l,W_l)\in B\right)\\ &\leqslant & \PP((Y,U)\in B))= 0 \end{eqnarray*}
as $l\to\infty$ as required.
\end{proof}
\begin{prop}\label{prop_HighDeficiencyProbability}
A random group with $n\geqslant 3$ generators and $m < n-1$ relators satisfies a minimum condition with high probability.
\end{prop}
\begin{proof}
Let $\rv{R}'$ be the set of $m$-tuples of cyclically reduced relators $r_1,\ldots, r_m$ of length $l$ such that no $\ttx_i$ vanishes in the abelianisation of the group
\[G=\gpp{\ttx_1,\ldots, \ttx_n}{r_1,\ldots,r_m} .\]
As remarked previously, $|\rv{R}'|\sim |\rv{R}|$ so that a relator lies in $\rv{R}'$ with high probability. Let $\rv{B}$ be the set of tuples $(r_1,\ldots, r_m)\in \rv{R}'$ such that $((r_1,\ldots, r_m),\phi)$ does not satisfy the minimum condition for any valid slope $\phi\in S(r_1,\ldots, r_m)$. Let $K>0$ and let ${\cal U}_{l,K}$ be the set of tuples in \rv{B} such that $|S/{\sim}|\geqslant K$.

We construct $K$ injections $\tau_1,\ldots, \tau_K\colon {\cal U}_{l,K}\to \rv[l+4]{R}'$ with disjoint images in the following way. Fix for all time some arbitrary well-ordering on $H^1(F_n,\Z)$.

For $(r_1,\ldots, r_m)\in {\cal U}_{l,K}$ let $\phi_1,\ldots, \phi_K$ be the first $K$ elements of $S$ representing different $\sim$-classes. For each $i$ and $j$ find the first $\phi_j$-minimal vertex of $r_i$ and insert at this point the commutator
\[\ttx_n^{-{\rm sign}(\phi_j(\ttx_n))}\ttx_i^{-{\rm sign}(\phi_j(\ttx_i))}\ttx_n^{{\rm sign}(\phi_j(\ttx_n))}\ttx_i^{{\rm sign}(\phi_j(\ttx_i))}\]
to form a new word $\tau_j(r_i)$. Note that the $\phi_j$-minimality of the insertion point implies that $\tau_j(r_i)$ is cyclically reduced. Insertion of commutators does not affect the abelianization, so the set $S$ is unchanged (though the equivalence relation $\sim$ will be altered).

We claim that the maps
\[\tau_j\colon {\cal U}_{l,K}\to \rv[l+4]{R}',\quad (r_1,\ldots,r_m)\mapsto (\tau_j(r_1),\ldots, \tau_j(r_m))\]
are injective and have disjoint images, whence $|\rv[l+4]{R}'|\geqslant K |{\cal U}_{l,K}|$. Observe that the relator $\tau_j(r_i)$ posesses a unique $\phi_j$-minimal vertex at the midpoint of the added commutator. Furthermore, if $\tau_j(r_i)$ has a unique $\psi$-minimal vertex $v$ for any $\psi\in S$, then since $r_i\in \rv{B}$ the vertex $v$ must be one of the three `new' vertices created by the addition of the commutator---and only one of these has the property that the four closest edges form a commutator. Hence there is a unique way to recover the relator $r_i$ from $\tau_j(r_i)$.

Finally take arbitrary $K\in\N$ and $\epsilon>0$. For $l$ sufficiently large (depending on $K$ and $\epsilon$), Lemma \ref{lem_ManyCorners} applies and we may estimate:
\[
\frac{|\rv{B}|}{|\rv{R}'|} = \frac{|{\cal U}_{l,K}|}{|\rv{R}'|} + \frac{|\rv{B}\smallsetminus {\cal U}_{l,K}|}{|\rv{R}'|}
 \leqslant  \frac{|\rv[l+4]{R}'|}{K|\rv{R}'|} + \epsilon \leqslant \frac{1}{K} (2n)^{4m} + \epsilon.
\]
Hence the probability that a tuple lies in \rv{B} becomes less than any prescribed positive value. It follows that with high probability a tuple of relators lies in the complement of \rv{B}---and thus satisfies a minimum condition.
\end{proof}

It only remains to assemble our various ingredients into the headline theorem.
\begin{proof}[Proof of Theorem \ref{thm_deficiencymore}]
By Proposition \ref{prop_HighDeficiencyProbability}, the random group satisfies a minimum condition. By Theorem \ref{thm_SmallCancWHP}, it is $C''(1/7)$ with high probability, whence by Theorem \ref{thm_C1/6consequences} it is two-dimensional and satisfies the Atiyah conjecture. Theorem \ref{thm_MinCondImpliesL2} now implies that $b^{(2)}_2(G)=0$; the value $b^{(2)}_1(G)=d$ follows from the Euler characteristic.

Furthermore, by Theorem \ref{thm_SmallCancEmbedding} the group $G$ embeds into a $C''(1/6)$, deficiency 1 group $H$ satisfying a minimum condition. By Theorems \ref{thm_ConsequenceOfL2} and Theorem \ref{thm_MinCondImpliesL2} we find that $H$ is virtually free-by-cyclic. Coherence of $G$ is inherited from $H$.
\end{proof}
\bibliographystyle{alpha}
\bibliography{Random_L2.bib}
\end{document}